\documentclass[english,reqno]{amsart}
\usepackage[T1]{fontenc}
\usepackage[latin9]{inputenc}
\usepackage{verbatim}
\usepackage{amsthm}
\usepackage{amstext}
\usepackage{amssymb}

\makeatletter
\numberwithin{equation}{section}
\numberwithin{figure}{section}
\theoremstyle{plain}
\newtheorem{thm}{\protect\theoremname}[section]
  \theoremstyle{definition}
  \newtheorem{defn}[thm]{\protect\definitionname}
  \theoremstyle{plain}
  \newtheorem{cor}[thm]{\protect\corollaryname}
  \theoremstyle{remark}
  \newtheorem{rem}[thm]{\protect\remarkname}

\usepackage{amscd}

\theoremstyle{plain}

\newcommand{\cw}{\stackrel{\mathcal{D}}{\rightharpoonup}}

\newcommand{\N}{\mathbb{N}}
\newcommand{\R}{\mathbb{R}}

\newcommand{\Z}{\mathbb{Z}}

\newcommand{\BV}{\dot{BV}(\mathbb{R}^N)}
\newcommand{\wlim}{\operatorname{w-lim}}

\makeatother

\usepackage{babel}
  \providecommand{\corollaryname}{Corollary}
  \providecommand{\definitionname}{Definition}
  \providecommand{\remarkname}{Remark}
\providecommand{\theoremname}{Theorem}

\begin{document}

\title[Defect of compactness]{Defect of compactness in spaces of bounded variation}

\author{Adimurthi}

\thanks{Research supproted by a visiting scholar grant from the Wenner-Gren
Foundations}

\address{TIFR-CAM, Sharadanagar, P.B. 6503, Bangalore 560065, India }

\email{aditi@math.tifrbng.res.in}

\author{Cyril Tintarev}

\address{Uppsala University, P.O.Box 480, 75 106 Uppsala, Sweden }

\email{tintarev@math.uu.se}


\begin{abstract}
Defect of compactness for non-compact imbeddings of Banach spaces
can be expressed in the form of a profile decomposition. Let $X$
be a Banach space continuously imbedded into a Banach space $Y$,
and let $D$ be a group of linear isometric operators on $X$. A profile
decomposition in $X$, relative to $D$ and $Y$, for a bounded sequence
$(x_{k})_{k\in\N}\subset X$ is a sequence $(S_{k})_{k\in\N}$, such
that $(x_{k}-S_{k})_{k\in\N}$ is a convergent sequence in $Y$, and,
furthermore, $S_{k}$ has the particular form $S_{k}=\sum_{n\in\N}g_{k}^{(n)}w^{(n)}$
with $g_{k}^{(n)}\in D$ and $w^{(n)}\in X$. %
This paper extends the profile decomposition proved by Solimini \cite{Solimini}
for Sobolev spaces $\dot{H}^{1,p}(\R^{N})$ with $1<p<N$ to the non-reflexive
case $p=1$. Since existence of ``concentration profiles'' $w^{(n)}$
relies on weak-star compactness, and the space $\dot{H}^{1,1}$ is
not a conjugate of a Banach space, we prove a corresponding result
for a larger space of functions of bounded variation. The result extends
also to spaces of bounded variation on Lie groups.
\end{abstract}
\maketitle

\section{\label{intro}Introduction }

In presence of a compact imbedding of a reflexive Banach space $X$
into another Banaxh space $Y$, Banach-Alaoglu theorem implies that
any bounded sequence in $X$ has a subsequence convergent in $Y$.
If the imbedding $X\hookrightarrow Y$ is continuous but not compact,
it may be possible to characterize a suitable subsequence as convergent
in $X$ once one subtracts a suitable ``defect of compactness'',
which typically, for sequences of functions, isolates the singular
behavior of the sequence. In broad sense this approach is known as
\emph{concentration compactness}, and in its more specific form, when
the defect of compactness is expressed as a sum of elementary concentrations,
is called profile decomposition. Profile decompositions were introduced
by Michael Struwe in 1984 for particular class of sequences in Sobolev
spaces. 
\begin{defn}
Profile decomposition of a sequence $(x_{k})$ in a reflexive Banach
space $X$, relative to a group $D$ of isometries of $X$, is an
asymptotic representation of $x_{k}$ as a convergent sum $S_{k}=\sum_{n\in\N}g_{k}^{(n)}w^{(n)}$
with $g_{k}^{(n)}\in D$, $w^{(n)}\in X$, such that $g_{k}(x_{k}-S_{k})\stackrel{}{\rightharpoonup}0$
for any sequence $(g_{k})\subset D$. In the latter case one says
that $x_{k}-S_{k}$ converges to zero $D$-weakly. 
\end{defn}
We refer the reader for motivation of profile decomposition as an
extension of the Banach-Alaoglu theorem, and a proof of both via non-standard
analysis, to Tao \cite{TaoBlog2}. For \emph{general} bounded sequences
in Sobolev spaces $\dot{H}^{1,p}(\R^{N})$, the profile decomposition,
relative to the group of translations and dilations, was proved in
\cite{Solimini}, and the $D$-weak convergence of the remainder was
identified as convergence in the Lorentz spaces $L^{p^{*},q}$, $q>p$,
where $p^{*}=\frac{pN}{N-p}$, and $1<p<N$ (which includes $L^{p^{*}}$
but excludes $L^{p^{*},p}$). The result of \cite{Solimini} was later
reproduced by Gérard \cite{Gerard} and Jaffard \cite{Jaffard}, who
extended it to the case of fractional Sobolev spaces, but, on the
other hand, gave a weaker form of remainder. For general Hilbert spaces,
equipped with a non-compact group of isometries of particular type,
existence of profile decomposition was proved in \cite{SchinTin}.
This, in turn, stimulated the search for new concentration mechanisms,
i.e. different groups $D$, that yield profile decompositions in concrete
functional spaces. In particular profile decompositions were proved
with inhomogeneous dilations $j^{-1/2}u(z^{j}$), $j\in\N$, with
$z^{j}$ denoting an integer power of a complex number, for problems
in the Sobolev space $H_{0}^{1,2}(B)$ of the unit disk, related to
the Trudinger-Moser functional; and with the action of the Galilean
invariance, together with shifts and rescalings, involved in the loss
of compactness in Strichartz imbeddings for the nonlinear Schrödinger
equations. For a more comprehensive summary of known profile decompositions,
including Besov and Triebel-Lizorkin spaces we refer the reader to
a recent survey \cite{survey}. Profile decomposition in the general,
uniformly convex and uniformly smooth, Banach space was proved recently
in \cite{ST2}, which required to introduce a new mode of convergence
of weak type (\cite{Part1}). Not unlike \cite{ST2}, this paper studies
profile decomposition by adapting the prior work on the topic to the
mode of convergence of weak type which is pertinent in the new setting.
A profile decomposition in the general non-reflexive Banach space
remains an open problem, and one should note that no profile decomposition
is possible when $p=\infty$, see e.g. a counterexample in \cite{ST2}. 

The space $L^{1}(\R^{N})$ lacks weak (or weak-star) sequential compactness:
indeed, consider a sequence of characteristic functions normalized
in $L^{1}$, $(\frac{1}{|A_{n}|}\chi_{A_{n}})$, where $\R^{N}\supset A_{1}\supset A_{2}\supset...$
are closed nested sets with $|\cap_{n\in\N}A_{n}|=0$ which has no
weakly convergent subsequence, while at the same time it converges
weakly in the sense of measures to the Dirac delta-function. This
suggests that when one studies a mapping on $L^{1}(\R^{N})$, it may
be beneficial to extend it to a larger domain, namely to the space
of finite signed measures, which it a conjugate of the Banach space
$C_{0}(\R^{N})$ and thus has the weak-{*} compactness property. Similarly,
it may be benefitial for a study of a mapping on the Sobolev space
$\dot{H}^{1,1}(\mathbb{R}^{N})$ to extend its domain to the space
of measurable functions whose weak derivative is a finite measure
(rather than necessarily a $L^{1}$- function), in other words, to
the space of functions of bounded variation. This space, $\BV$ contains,
of course, functions that are qualitatively different from those in
$H^{1,1}$. For example, the characteristic function of a ball belongs
to $\BV$ and has a disconnected range $\{0,1\}$, while every element
in $\dot{H}^{1,1}(\mathbb{R}^{N})$ is represented by a function with
a connected range.

The space of functions of bounded variations is of particular importance
in geometric measure theory and in image analysis and has been a subject
of intense scholarly interest. Our arguments generally follow the
previous proofs of profile decompositions (\cite{ccbook,Solimini})
with adjustments to the peculiarities of the space. In Section 2 we
summarize some known properties of the functions of bounded variation.
In Section 3 we prove that imbedding into $L^{N/(N-\text{1)}}$ is
cocompact relative to the group of dyadic dilations and translations,
which allows us in Section 4 to get a profile decomposition with the
remainder vanishing in $L^{N/(N-\text{1)}}$. In Section 5 we give
some applications for minimizations of functionals that complement
the applications in the paper of Bartsch and Willem \cite{BartschWillem},
Theorem \ref{thm:5.2} in particular, as it requires a cocompactness
argument. In Section 6 we generalize the profile decomposition of
Section 4 to the case of subelliptic operators on nilpotent Lie groups.
The main results of the paper are Theorem \ref{thm:coco}, Theorem
\ref{thm:pd}, their Lie group generalizations Theorem \ref{thm:coco-1}
and Theorem \ref{thm:pd-1}, and results on existence of minimizers
Theorem \ref{thm:inf} and Theorem \ref{thm:5.2}.

\section{Space $\BV$}

We summarize here some known properties of the space of functions
with bounded variation. For a comprehensive exposition of the subject
we refer the reader to the book \cite{Ambrosio}. We assume throughout
the paper that $N\ge2$.
\begin{defn}
\label{def:bv}The space of functions of bounded variation $\dot{BV}(\mathbb{R}^{N})$
is the space of all measurable functions $u:\R^{N}\to\R$ vanishing
at infinity (i.e. $\forall\epsilon>0$ $|\{x\in\R^{N}:\;|u(x)|>\epsilon\}|<\infty$)
such that 
\begin{equation}
\|Du\|:=\sup_{v\in C_{0}^{\infty}(\R^{N};\R^{N}):\|v\|_{\infty}=1}\int_{\R^{N}}u\,\mathrm{div}\, v<\infty.\label{eq:bova}
\end{equation}

\end{defn}
The $\dot{BV}(\mathbb{R}^{N})$-norm can be interpreted as the total
variation $\|Du\|$ of the measure associated with the derivative
$Du$ (in the sense of distributions on $\R^{N}$). If $u\in C_{0}^{1}(\mathbb{R}^{N})$,
then the right hand side in (\ref{eq:bova}) by integration by parts
equals $\int|\nabla u|$. The value of the total variation of $Du$
on a measurable set $A\subset\R^{N}$ will be denoted as $\|Du\|_{A}$.

The space $\dot{BV}(\mathbb{R}^{N})$ is a conjugate space and therefore
is complete. We will follow the convention that calls the weak-star
convergence in the space of bounded variation \emph{weak convergence}.
It is well-known that $\dot{BV}(\mathbb{R}^{N})$ is separable and
therefore each bounded sequence in $\dot{BV}(\mathbb{R}^{N})$ has
a weakly (i.e. weakly-star) convergent subsequence.
\begin{defn}
\label{def:wc}A sequence $(u_{k})\subset\dot{BV}(\mathbb{R}^{N})$
is said to converge weakly to $u$ if $u_{k}\to u$ in $L_{\mathrm{loc}}^{1}(\R^{N})$
and the weak derivatives $\partial_{i}u_{k}$, $i=1,\dots,N$ converge
to $\partial_{i}u$ weakly as finite measures on $\R^{N}$. 
\end{defn}
We will need the following properties of $\dot{BV}(\mathbb{R}^{N})$.
\begin{enumerate}
\item \emph{Invariance.} The group of operators on $\BV$, 
\begin{equation}
D=\{g[j,y]:\; u\mapsto2^{(N-1)j}u(2^{j}(\cdot-y)\}{}_{j\in\Z,y\in\R^{N}},\label{eq:gauge}
\end{equation}
consists of linear isometries of $\dot{BV}(\mathbb{R}^{N})$, which
are also linear isometries on $L^{\frac{N}{N-1}}(\R^{N})$.
\item Density of $C_{0}^{\infty}(\R^{N})$ in \emph{strict} topology (the
closure of $C_{0}^{\infty}(\R^{N})$ in the norm topology is $\dot{H}^{1,1}$).
One says that $u_{k}$ converges strictly to $u$ if $\|u_{k}-u\|_{1^{*}}\to0$
and $\|Du_{k}\|\to\|Du\|$.
\item V.Maz'ya's inequality (often referred to as Sobolev, Aubin-Talenti
or Gagliardo-Nirenberg inequality) \cite{Mazya}
\begin{equation}
NV_{N}^{1/N}\|u\|_{1^{*}}\le\|Du\|,\label{eq:imbedding}
\end{equation}
where $1^{*}=\frac{N}{N-1}$ and $V_{N}$ is the volume of the unit
ball in $\R^{N}$. A local version of this inequality is 
\[
\|u\|_{1^{*},\Omega}\le C(\|Du\|_{\Omega}+\|u\|_{1,\Omega}),
\]
where $\Omega\subset\R^{N}$ is a bounded domain with sufficiently
regular, say locally $C^{1}$-boundary.
\item Hardy inequality :$ $
\[
\|Du\|\ge(N-1)\int_{\R^{N}}\frac{|u|}{|x|}dx\,.
\]
(It follows from the Hardy inequality in $\dot{H}^{1,1}$ and the
density of $C_{0}^{\infty}$ in $\BV$ with respect to the strict
convergence, if one first replaces $1/|x|$ with its $L^{N}$-approximations
from below.)
\item Local compactness: for any set $\Omega\subset\R^{N}$ of finite Lebesgue
measure, $\dot{BV}(\mathbb{R}^{N})$ is compactly imbedded into $L^{1}(\Omega)$. 
\item Chain rule (a simplified version of a more elaborate statement due
to Vol'pert, see \cite[Remark 3.98]{Ambrosio}): let $\varphi\in C^{1}(\R)$.
Then for every $u\in\BV,$ 
\begin{equation}
\|D\varphi(u)\|\le\|\varphi'\|_{\infty}\|Du\|.\label{eq:ChainRule}
\end{equation}

\end{enumerate}

\section{Cocompactness of the imbedding $\dot{BV}(\mathbb{R}^{N})\protect\hookrightarrow L^{1^{*}}(\mathbb{R}^{N})$}
\begin{defn}
Let $X$ be a Banach space and let $D$ be a group of linear isometries
of $X$. One says that a sequence $(x_{k})\subset X$ is \emph{$D$-vanishing}
(to be written $x_{k}\cw0$) if for any sequence $(g_{k})\subset D$
one has $g_{k}x_{k}\rightharpoonup0.$ A continuous imbedding of $X$
into a topological space $Y$ is called cocompact with respect to
$D$ if $x_{k}\cw0$ implies $x_{k}\to0$ in $Y$. 
\end{defn}
We extend this definition to $\BV$ by understanding weak convergence
in the sense of Definition \ref{def:wc}. 

The proof of the theorem below repeats much of the proof of Lemma
5.3 in \cite{ccbook}, but uses different argument for the evaluation
of sums of BV-seminorms over lattices. 
\begin{thm}
\label{thm:coco}The imbedding $\dot{BV}(\mathbb{R}^{N})\hookrightarrow L^{1^{*}}(\mathbb{R}^{N})$
is cocompact relative to the group (\ref{eq:gauge}), i.e. if, for
any sequence $(j_{k},y_{k})\subset\Z\times\R^{N}$, $g[j_{k},y_{k}]u_{k}\rightharpoonup0$
then $u_{k}\to0$ in $L^{1^{*}}(\R^{N})$.\end{thm}
\begin{proof}
Let $(u_{k})\subset\BV$ be such that for any $(j_{k},y_{k})\subset\Z\times\R^{N}$
, $g[j_{k},y_{k}]u_{k}\rightharpoonup0$. 

1. Assume first that $\sup_{k\in\N}\|u_{k}\|_{\infty}<\infty$ and
$\sup_{k\in\N}\|u_{k}\|_{1,\R{}^{N}}<\infty$. Then, using the $L^{\infty}$-boundedness
of $(u_{k})$ we have 
\[
\int_{(0,1)^{N}}|u_{k}|^{1^{*}}\le C\left(\|Du_{k}\|_{(0,1)^{N}}+|\int_{(0,1)^{N}}u_{k}|\right)\left(\int_{(0,1)^{N}}|u_{k}|\right)^{1^{*}-1}.
\]
Repeating this inequality for the domain of integration $(0,1)^{N}+y$,
$y\in\Z^{N}$, and adding the resulting inequalities over all $y\in\Z^{N}$,
we have 
\begin{equation}
\int_{\R^{N}}|u_{k}|^{1^{*}}\le C(\|Du_{k}\|_{\R{}^{N}}+\|u_{k}\|_{1,\R{}^{N}})\left(\sup_{y\in\Z^{N}}\int_{(0,1)^{N}}|u_{k}(\cdot-y)|\right)^{1^{*}-1}.\label{eq:ineq1}
\end{equation}
Here we use the fact that the sum $\sum_{y\in\Z^{N}}\|Du_{k}\|_{(0,1)^{N}+y}$
can be split into $3^{N}$ sums of variations over unions of cubes
with disjoint closures, each of them, as follows from Definition \ref{def:bv}
bound by $\|Du_{k}\|_{\R{}^{N}}$, which implies $\sum_{y\in\Z^{N}}\|Du_{k}\|_{(0,1)^{N}+y}\le3^{N}\|Du_{k}\|_{\R{}^{N}}$.

The last term in (\ref{eq:ineq1}) converges to zero, since by the
assumption $g[j_{k},y_{k}]u_{k}\rightharpoonup0$ we have $u_{k}(\cdot-y_{k})\to0$
in $L^{1}((0,1)^{N})$ for any sequence $(y_{k})\subset\R^{N}$. 

2. We now abandon the restrictions imposed in the previous step on
the sequence $(u_{k})$. Let $\chi\in C_{0}^{\infty}((\frac{1}{2^{N-1}},4^{N-1}))$
be such that $\chi(t)=t$ whenever $t\in[1,2^{N-1}]$. Let $\chi_{j}(t)=2^{(N-2)j}\chi(2^{-(N-1)j}|t|)$,
$j\in\Z$, $t\in\R$, and note that $\|\chi'_{j}\|_{\infty}=\|\chi'\|_{\infty}$.
Consider now a general sequence $(u_{k})\subset\BV$ satisfying $g[j_{k},y_{k}]u_{k}\rightharpoonup0$
for any $(j_{k},y_{k})\subset\Z\times\R^{N}$ . By (\ref{eq:imbedding})
we have 
\[
\int\chi_{j}(u_{k})^{1^{*}}\le C\|D\chi_{j}(u_{k})\|\left(\int\chi_{j}(u_{k})^{1^{*}}\right)^{1^{*}-1}.
\]
Let us sum up the inequalities over $j\in\Z$. Note that by (\ref{eq:ChainRule})
$\|D\chi_{j}(u_{k})\|\le\|\chi'\|_{\infty}\|Du_{k}\|_{A_{kj}}$ where
$A_{kj}=\{x\in\R^{N}:\;|u_{k}|\in(2^{(j-1)(N-1)},2^{(j+2)(N-1)})\}$.
Furthermore, one can break all the integers $j$ into six disjoint
sets $J_{1},\dots,J_{6}$ , such that, for any $m\in\{1,2,3,4,5,6\}$,
all functions $\chi_{j}(u_{k})$, $j\in J_{m}$, have pairwise disjoint
supports. Consequently, $\sum\|Du_{k}\|_{A_{kj}}\le6\|Du_{k}\|$.
We have therefore 
\[
\int_{\R^{N}}|u_{k}|^{1^{*}}\le C\|Du_{k}\|\sup_{j\in\Z}\left(\int\chi_{j}(u_{k})^{1^{*}}\right)^{1^{*}-1}.
\]
It suffices now to show that for any sequence $(j_{k})\subset\Z$,
$\chi_{j_{k}}(u_{k})\to0$ in $L^{1^{*}}$. Taking into account the
invariance of the $L^{1^{*}}$-norm under operators $g[j,y]$, it
suffices to show that $\chi(2^{j_{k}(N-1)}|u_{k}(2^{j_{k}}\cdot)|)\to0$
in $L^{1^{*}}$, but this is immediate from the assumption $g[j_{k},y_{k}]u_{k}\rightharpoonup0$
and the argument of the step 1, once we take into account that for
sequences uniformly bounded in $L^{\infty},$ $L^{1^{*}}$-convergence
follows from $L^{1}$ convergence.\end{proof}
\begin{cor}
\label{thm:coco-2}The imbedding $\dot{H}^{1,1}(\mathbb{R}^{N})\hookrightarrow L^{1^{*}}(\mathbb{R}^{N})$
is cocompact with respect to the group (\ref{eq:gauge})
\end{cor}

\section{Profile decomposition}
\begin{thm}
\label{thm:pd}Let $(u_{k})\subset\BV$ be a bounded sequence. For
each $n\in\N$ there exist $w^{(n)}\in\BV$, and sequences $(j_{k}^{(n)},y_{k}^{(n)})\subset\Z\times\R^{N}$
with $j_{k}^{(1)}=0$ $y_{k}^{(1)}=0$, satisfying 
\[
|j_{k}^{(n)}-j_{k}^{(m)}|+|y_{k}^{(n)}-y_{k}^{(m)}|\to\infty\text{ whenever }m\neq n,
\]
such that for a renumbered subsequence, $g[-j_{k}^{(n)},-y_{k}^{(n)}]u_{k}\rightharpoonup w^{(n)}$,
as $k\to\infty$,
\begin{equation}
r_{k}\stackrel{\mathrm{def}}{=}u_{k}-\sum_{n}g[j_{k}^{(n)},y_{k}^{(n)}]w^{(n)}\to0\text{ in }L^{\frac{N}{N-1}}(\mathbb{R}^{N}),\label{eq:vanish}
\end{equation}
where the series $\sum_{n}g[j_{k}^{(n)},y_{k}^{(n)}]w^{(n)}$ converges
in $\BV$ uniformly in $k$, and
\begin{equation}
\sum_{n\in\N}\|Dw^{(n)}\|+o(1)\le\|Du_{k}\|\le\sum_{n\in\N}\|Dw^{(n)}\|+\|Dr_{k}\|+o(1).\label{eq:norms}
\end{equation}
\end{thm}
\begin{proof}
Without loss of generality we may assume that $u_{k}\rightharpoonup0$
(otherwise, one may pass to a weakly convergent subsequence and subtract
the weak limit). Observe that if $u_{k}\cw0$, the theorem is proved
with $r_{k}=u_{k}$ and $w^{(n)}=0$, $n\in\N$. Otherwise consider
the expressions of the form $w^{(1)}=\wlim g[-j_{k}^{(1)},-y_{k}^{(1)}]u_{k}$.
The sequence $u_{k}$ is bounded, $D$ is a set of isometries, so
the sequence $g[-j_{k}^{(1)},-y_{k}^{(1)}]u_{k}$ has a weakly convergent
subsequence. Since we assume that $u_{k}$ is not D-vanishing, there
exists necessarily a sequence $(j_{k}^{(1)},y_{k}^{(1)})$ such that,
evaluated on a suitable subsequence, $w^{(1)}\neq0$. Let $v_{k}^{(1)}=u_{k}-g[j_{k}^{(1)},y_{k}^{(1)}]w^{(1)}$,
and observe that $g[-j_{k}^{(1)},-y_{k}^{(1)}]v_{k}^{(1)}=g[-j_{k}^{(1)},-y_{k}^{(1)}]u_{k}-w^{(1)}\rightharpoonup0.$
If $v_{k}^{(1)}\cw0$, the assertion of the theorem is verified with
$r_{k}=v_{k}^{(1)}$. If not - we repeat the argument above - there
exist, necessarily, a sequence $(j_{k}^{(1)},y_{k}^{(1)})$ and a
$w^{(2)}\neq0$ such that, on a renumbered subsequence, $w^{(2)}=\wlim g[-j_{k}^{(2)},-y_{k}^{(2)}]v_{k}^{(1)}$.
Let us set $v_{k}^{(2)}=v_{k}^{(1)}-g[j_{k}^{(2)},y_{k}^{(2)}]w^{(2)}$.
Then we will have
\[
g[-j_{k}^{(2)},-y_{k}^{(2)}]v_{k}^{(2)}=g[-j_{k}^{(2)},-y_{k}^{(2)}]v_{k}^{(1)}-w^{(2)}\rightharpoonup0.
\]
If we assume that $g[-j_{k}^{(1)},-y_{k}^{(1)}]g[j_{k}^{(2)},y_{k}^{(2)}]w^{(2)}\not\rightharpoonup0$,
or, equivalently, that $|j_{k}^{(1)}-j_{k}^{(2)}|+|y_{k}^{(1)}-y_{k}^{(2)}|$
has a bounded subsequence, then, passing to a renamed subsequence
we will have $g[-j_{k}^{(1)},-y_{k}^{(1)}]g[j_{k}^{(2)},y_{k}^{(2)}]\to g[j_{0},y_{0}]$
in the sense of strong operator convergence, for some $j_{0}\in\Z$,
$y_{0}\in\R^{N}$. Then 
\[
w^{(2)}=\wlim g[-j_{k}^{(2)},-y_{k}^{(2)}]v_{k}^{(1)}=\wlim(g[-j_{k}^{(2)},-y_{k}^{(2)}]g[j_{k}^{(1)},y_{k}^{(1)}])g[-j_{k}^{(1)},-y_{k}^{(1)}]v_{k}^{(1)}
\]
\[
=\wlim(g[-j_{0},-y_{0}]g[-j_{k}^{(1)},-y_{k}^{(1)}]v_{k}^{(1)}=0,
\]
a contradiction that proves that $g[-j_{k}^{(1)},-y_{k}^{(1)}]g[j_{k}^{(2)},y_{k}^{(2)}]\rightharpoonup0$,
or, equivalently, $|j_{k}^{(1)}-j_{k}^{(2)}|+|y_{k}^{(1)}-y_{k}^{(2)}|\to\infty$.
Then we also have $g[-j_{k}^{(2)},-y_{k}^{(2)}]g[j_{k}^{(1)},y_{k}^{(1)}]\rightharpoonup0$. 

Recursively we define:
\[
v_{k}^{(n)}=v_{k}^{(n-1)}-g[j_{k}^{(n)},y_{k}^{(n)}]w^{(n)}=u_{k}-g[j_{k}^{(1)},y_{k}^{(1)}]w^{(1)}-\dots-g[j_{k}^{(n)},y_{k}^{(n)}]w^{(n)},
\]
 where $w^{(n)}=\wlim g[-j_{k}^{(n)},-y_{k}^{(n)}]v_{k}^{(n-1)}$,
calculated on a successively renumbered subsequence. We subordinate
the choice of $(j_{k}^{(n)},y_{k}^{(n)})$, and thus the extraction
of a subsequence for every given $n$, to the following requirements.
For every $n\in\N$ we set 
\[
W_{n}=\{w\in\BV\setminus\{0\}:\;\exists(j_{k},y_{k})\subset\Z\times\R^{N},(k)\subset\N^{\N}:g[-j_{m},-y_{m}]v_{k_{m}}^{(n)}\rightharpoonup w\},
\]
and 
\[
t_{n}=\sup_{w\in W_{n}}\|Dw\|.
\]
 Note that $t_{n}\le\sup\|u_{k}\|<\infty$. If for some $n$, $t_{n}=0$,
the theorem is proved with $r_{k}=v_{k}^{(n-1)}$. Otherwise, we choose
a $w^{(n+1)}\in W_{n}$ such that $\|Dw^{(n+1)}\|\ge\frac{1}{2}t_{n}$
and the sequence $(j_{k}^{(n+1)},y_{k}^{(n+1)})$ is chosen so that
on a subsequence that we renumber, $g[(-j_{k}^{(n+1)},-y_{k}^{(n+1)}]v_{k}^{(n)}\rightharpoonup w^{(n+1)}$.
An argument analogous to the one brought above for $n=1$ shows that
$g[(-j_{k}^{(p)},-y_{k}^{(p)}]g[(j_{k}^{(q)},y_{k}^{(q)}]\rightharpoonup0$,
or, equivalently, 
\begin{equation}
|j_{k}^{(p)}-j_{k}^{(q)}|+|y_{k}^{(p)}-y_{k}^{(q)}|\to\infty\label{eq:sep}
\end{equation}
whenever $p\neq q,\, p,q\leq n.$ 

Let us show the lower bound inequality in (\ref{eq:norms}). Let $n\in\N$
and let $(j_{k}^{(i)},y_{k}^{(i)})_{k}$, $w^{(i)}$ and $(v_{k}^{(i)})_{k}$,
$i=1,\dots,n$, be defined as above. Let $v^{(i)}\in C_{0}^{\infty}(\R^{N})$,
$\|v^{(i)}\|_{\infty}\le1$ $i=1,\dots,n$, and set $S_{k}^{(n)}=\sum_{i=1}^{n}g[j_{k}^{(i)},y_{k}^{(i)}]w^{(i)}$,
$V_{k}^{(n)}=\sum_{i=1}^{n}2^{(1-N)j_{k}^{(i)}}g[j_{k}^{(i)},y_{k}^{(i)}]v^{(i)}$.
(To clarify the construction: the operator $2^{(1-N)j/2}g[j,y]$ is
the $L^{2}(\R^{N})$-adjoint of $g[-j,-y]$. ) Then, noting that $\|V_{k}^{(n)}\|_{\infty}\le1$
and taking into account (\ref{eq:sep}),we have 
\[
\|Du_{k}\|\ge\int v_{k}^{(n)}\mathrm{div}V_{k}^{(n)}+\int S_{k}^{(n)}\mathrm{div}V_{k}^{(n)}
\]
\[
=\sum_{i=1}^{n}\int g[(-j_{k}^{(i)},-y_{k}^{(i)}]v_{k}^{(n)}v^{(i)}+\sum_{i=1}^{n}w^{(i)}\mathrm{div}v^{(i)}.
\]
Since the first term converges to zero by construction, while $v^{(i)}$
is arbitrary, we have $\|Du_{k}\|\ge\sum_{i=1}^{n}\|Dw^{(i)}\|+o_{k\to\infty}(1)$.
Since $n$ is arbitrary, the lower bound in (\ref{eq:norms}) follows.

Note now that %
{} $\sum_{i=1}^{\infty}t_{i}\le2\|Du_{k}\|+o(1)$. Furthermore, $\|DS_{k}^{(n)}\|\le\sum_{i=1}^{n}t_{i}+o(1),$
and on a suitable subsequence we have $\|DS_{k}^{(n)}\|\le2\sum_{i=1}^{n}t_{i}$,
and furthermore the inequality remains true even if one omits an arbitrary
subset of terms in the sum $S_{k}^{(n)}$. Consequently, by an elementary
diagonalization argument, on a suitable subsequence the series $S_{k}^{\infty}$
converges in $\BV$ uniformly in $k$. This together with (\ref{eq:sep})
implies that $u_{k}-S_{k}^{\infty}\stackrel{D}{\rightharpoonup0},$
which by Theorem \ref{thm:coco} implies (\ref{eq:vanish}). Finally,
the second inequality in (\ref{eq:norms}) follows from convergence
of $S_{k}^{\infty}$ and the triangle inequality for norms.
\end{proof}

\section{Sample minimization problems}

Let $a>0$ be such that $w:=a\chi_{B}$, where $B$ is a unit ball,
is a maximizer for the problem
\[
c_{0}=\sup_{u\in\BV:\|Du\|=1}\int_{\R^{N}}|u|^{1^{*}}.
\]
By scaling invariance, $w_{R}=R^{1-N}a\chi_{B_{R}}$ is then also
a maximizer. In the following, generally non-compact, problem the
existence of minimizers is proved by means of specific properties
of the space $\BV$ rather than concentration argument.
\begin{thm}
\label{thm:inf}Let function $F\in C(\R)$ be such that the following
supremum is poisitve and is attained:
\begin{equation}
0<m=\sup_{s\in\R}F(s)/|s|^{1^{*}}=F(t)/|t|^{1^{*}}\text{ for some }\; t\in\R.\label{eq:iii-1}
\end{equation}
 Then the maximum in the relation
\[
c=\sup_{u\in\BV:\|Du\|=1}\int_{\R^{N}}F(u).
\]
is attained at the point $w_{R}$ with $R=(\frac{a}{t})^{\frac{1}{N-1}}$.\end{thm}
\begin{proof}
Since $F(u)\le m|u|^{1^{*}}$, we have $c\le mc_{0}$. On the other
hand, comparing the supremum with the value of the functional at $w_{R}$
we have $c\ge\int F(w_{R})=F(t)|B_{R}|=m|t|^{1^{*}}|B_{R}|=m\int|aR^{1-N}\chi_{B_{R}}|^{1^{*}}=mc_{0}$.
Therefore $c=mc_{0}$ and is attained at $w_{R}$.\end{proof}
\begin{thm}
\label{thm:5.2}Let $0<\lambda<N-1$. Then the minimum in\textup{ }

\textup{
\[
\kappa=\inf_{u\in\BV:\int_{\R^{N}}|u|^{1^{*}}=1}\|Du\|-\lambda\int\frac{|u|}{|x|}dx
\]
is attained. }\end{thm}
\begin{proof}
The proof of the argument is a standard use of profile decomposition
and may be abbreviated. Let $(u_{k})$ be a minimizing sequence. Applying
Theorem \ref{thm:pd} and noting that there exists a subset of indices
$I\subset\N$ such that $\int\frac{|u_{k}|}{|x|}dx\to\sum_{n\in I}\int\frac{|w^{(n)}|}{|x|}dx$
(provided that the functions $w^{(n)}$ are redefined, as it is always
possible, by appication of constant operator $g[j_{n,}y_{n}]\subset D$),
we have, using the notation 
\[
J(u)=\|Du\|-\lambda\int\frac{|u|}{|x|}dx,
\]
and recalling (\ref{eq:norms}), 
\begin{equation}
J(u_{k})\ge\sum_{n\text{\ensuremath{\in}I}}J(w^{(n)})+\sum_{n\not\in I}\|Dw^{(n)}\|+o(1).\label{eq:a}
\end{equation}
On the other hand, from Brezis-Lieb lemma follows 
\begin{equation}
\int_{\R^{N}}|u_{k}|^{1^{*}}=\sum_{n\in\N}\int_{\R^{N}}|w^{(n)}|^{1^{*}}+o(1).\label{eq:b}
\end{equation}
Moreover, each $w^{(n)}$ necessarily minimizes the respective functional
($J$ if $n\in I$, $\|D\cdot\|$ if $n\notin I$) over the functions
$u\in\BV$ satisfying $\int_{\R^{N}}|u|^{1^{*}}=\int_{\R^{N}}|w^{(n)}|^{1^{*}}$.
In particular, $w^{(n)}$ for $n\notin I$ are multiples of the characteristic
function of some ball, which are clearly not minimizers, up to normalization,
for the functional $J$. 

From the standard convexity argument, relations (\ref{eq:a}) and
(\ref{eq:b}) imply that, necessarily, $w^{(n)}=0$ for all $n\in\N$
except $n=m$ with some $m\in I$. Thus, $\int_{\R^{N}}|w^{(m)}|^{1^{*}}=\int_{\R^{N}}|u_{k}|^{1^{*}}=1$
and $J(w^{(m)})\le J(u_{k})=\kappa+o(1)$. This implies that $w^{(m)}$
is a minimizer.
\end{proof}

\section{6.1 Profile decomposition in $\dot{BV}(G)$ for Carnot groups}

The space of bounded variations on stratified nilpotent Lie groups,
often called Carnot groups, has been studied in detail by Garofalo
and Nhieu \cite{Garofalo} and we first summarize relevant definitions
and properties from that paper. The underlying Sobolev space is defined
on a Carnot group G by a set of vector fields $\{X_{j}\}_{j=1,\dots n}$
which satisfy the Hörmander condition. More specifically, vectors
$\{X_{j}\}_{j=1,\dots n}$ span the first stratum $Y_{1}$ of the
associated Lie algebra, their commutators spans the second stratum
$Y_{2}$, and further successive commutations define furthersuccessive
strata. Since the group is nilpotent, there is a minimal number $m\ge1$
such that $Y_{m+1}=\{0\}$ and Hörmander condition is equivalent to
$Y_{1}\dots Y_{m}$ spanning the whole Lie algebra. The left shift
invariant Haar measure on such groups coincides with the Lebesgue
measure. The case $m=1$ is the Euclidean case and the most commonly
occurring example in literature is the Heisenberg group corresponding
to $N=3$, $n=2$, and $m=2$. The subelliptic Sobolev space $\dot{H}^{1,1}(G)$
, $n>1$, is the space of measurable functions such that 
\[
\|u\|_{1,1}=\int_{G}\sum_{j=1}^{n}|X_{j}u|d\lambda.
\]
The related space of bounded variations $\dot{BV}(G)$ is defined
by the norm 
\begin{equation}
\|Du\|_{G}=\sup_{v\in C_{0}^{1}(G;\R^{N}):\;\|v\|_{\infty}\le1}\int_{G}u\sum_{j=1}^{n}X_{j}^{*}v_{j}\; d\lambda.\label{eq:dug}
\end{equation}
There is a continuous imbedding $\dot{BV}(G)\hookrightarrow L^{1^{*}}(G)$
where $1^{*}=\frac{Q}{Q-1}$ , where $Q=\sum_{i=1}^{m}i\dim Y_{i}$.

A seminorm $\|Du\|_{\Omega}$, where $\Omega\subset G$ is a Lebesgue-measurable
set, is defined by the expression (\ref{eq:dug}) with the integration
over $\Omega$ instead of $G$. When $\Omega$ is a bounded domain
with Lipschitz boundary, $\|Du\|_{\Omega}+\int_{\Omega}|u|$ is a
norm defining the subelliptic Sobolev space $BV(\Omega)$ which is
compactly imbedded into $L^{1}(\Omega)$.

The simplified chain rule in $\dot{BV}(G)$, given a function $f\in C^{1}(\R)$,
is the inequality 
\[
\|Df(u)\|_{\Omega}\le\|f'\|_{\infty}\|Du\|_{\Omega}.
\]

We define the group $D_{G}$ of isometries on $\dot{BV}(G)$ as a
product group of left shifts by $G$, $u\mapsto u\circ\eta$, $\eta\in G$,
and of discrete dilations 
\[
h_{j}=u\mapsto2^{(Q-1)j}u\circ\exp\circ\delta_{2^{j}}\circ\exp^{-1},\; j\in\Z,
\]
 where the anisotropic dilations $\delta_{t}$, $t>0$, map the variables
$(y_{1},\dots,y_{m})\in Y_{1}\times\dots\times Y_{m}$ of the stratified
Lie algebra of $G$ into $ty_{1},\dots t^{m}y_{m})$. 

We have
\begin{thm}
\label{thm:coco-1}Let $G\backsimeq\R^{N}$, $N\ge2$, be a stratified
nilpotent Lie group. The imbedding $\dot{BV}(G)\hookrightarrow L^{\frac{Q}{Q-1}}(G)$
is cocompact with respect to the operator group $D_{G}$ . More specifically,
if for any sequence $(j_{k},\eta_{k})\subset\Z\times G$, $(h_{j_{k}}u_{k})\circ\eta_{k}\rightharpoonup0$
then $u_{k}\to0$ in $L^{\frac{Q}{Q-1}}(G)$.
\end{thm}
The proof is a straightforward combination of the proof of Theorem
\ref{thm:coco} and the proof of an analogous statement for $\dot{H}^{1,2}(G)$
(\cite[Lemma 9.4]{ccbook}, once notes the need in the covering lemma
\cite[Lemma A1]{ccbook}.
\begin{thm}
\label{thm:pd-1}Let $G\backsimeq\R^{N}$, $N\ge2$, be a stratified
nilpotent Lie group. Let $(u_{k})\subset\dot{BV}(G)$ be a bounded
sequence. For each $n\in\N$ there exist $w^{(n)}\in\dot{BV}(G)$,
and sequences $(j_{k}^{(n)},\eta_{k}^{(n)})\subset\Z\times G$ with
$j_{k}^{(1)}=0$, $\eta_{k}^{(1)}=e$, satisfying 
\[
|j_{k}^{(n)}-j_{k}^{(m)}|+|\exp^{-1}(\eta_{k}^{(n)}\circ\eta_{k}^{(m)^{-1}})|\to\infty\text{ whenever }m\neq n,
\]
such that for a renumbered subsequence, $(h_{-j_{k}}u_{k})\circ\eta_{k}^{-1}\rightharpoonup w^{(n)}$,
as $k\to\infty$,
\begin{equation}
r_{k}\stackrel{\mathrm{def}}{=}u_{k}-\sum_{n}h_{j_{k}^{(n)}}(w^{(n)}\circ\eta_{k}^{(n)})\to0\text{ in }L^{\frac{Q}{Q-1}}(G),\label{eq:vanish-1}
\end{equation}
where the series $\sum_{n}h_{j_{k}^{(n)}}(u_{k}\circ\eta_{k}^{(n)})w^{(n)}$
converges in $\dot{BV}(G)$ uniformly in $k$, and
\begin{equation}
\sum_{n\in\N}\|Dw^{(n)}\|+o(1)\le\|Du_{k}\|\le\sum_{n\in\N}\|Dw^{(n)}\|+\|Dr_{k}\|+o(1).\label{eq:norms-1}
\end{equation}

\end{thm}
The proof is a straightforward modification of the proof of Theorem
\ref{thm:pd} on the lines of \cite[Remark 9.3]{ccbook}. 
\begin{rem}
As we already mentioned, there is no profile decomposition in $L^{\infty}$,
while the existing proof of profile decomposition in Banach spaces
(in \cite{ST2}) is based on a bound on the norms of profiles in the
form 
\[
\sum_{n}\delta(\|w^{(n)}\|/\|u_{k}\|)\le1
\]
that involves the modulus of convexity of the space, and it is not
known in general which non-reflexive spaces admit profile decompositions.
Is possible, however, in presence of an imbedding into a uniformly
convex space (such as $L^{1^{*}}$ in the present paper) to write
a profile decomposition in terms of the target space, but this yields
convergence of the series representing defect of compactness only
in the weaker norm of the target space. 

\textsc{Acknowledgement}. The authors thank Rupert Frank for his kind
remark concerning Section 5.\end{rem}

\end{document}